\documentclass[11pt,twoside]{article}
\usepackage[left=3cm,right=3cm,top=3cm,bottom=3cm]{geometry}
\usepackage{amssymb}
\usepackage{graphicx}
\usepackage{graphicx,type1cm,eso-pic,color}
\usepackage{amssymb}
\usepackage{amssymb,amsmath}
\usepackage{graphicx,epsfig}
\usepackage{enumerate}
\usepackage{color}
\usepackage {multicol}
\usepackage{lineno}
\numberwithin{equation}{section}
\newtheorem{theorem}{Theorem}
\newtheorem{lemma}{Lemma}

\newtheorem{definition}{Definition}
\newtheorem{proposition}{Proposition}

\newtheorem{corollary}{Corollary}
\newtheorem{example}{Example}

\makeatletter
\AddToShipoutPicture{
	\setlength{\@tempdimb}{.5\paperwidth}
	\setlength{\@tempdimc}{.5\paperheight}
	\setlength{\unitlength}{1pt}
	\put(\strip@pt\@tempdimb,\strip@pt\@tempdimc)
}
\makeatother

\usepackage{scalerel,stackengine}
\stackMath
\newcommand\reallywidehat[1]{%
	\savestack{\tmpbox}{\stretchto{%
			\scaleto{%
				\scalerel [\widthof{\ensuremath{#1}}]{\kern-.6pt\bigwedge\kern-.6pt}%
				{\rule[-\textheight/2]{1ex}{\textheight}}
			}{\textheight}%
		}{0.5ex}}%
	\stackon[1pt]{#1}{\tmpbox}%
}
\begin{document}
	\setcounter{page}{1}
	\thispagestyle{empty}
	\pagestyle{myheadings}
	\markboth{}{ }
	\date{}
	\vspace{.1in}
	\begin{center}
		{\Large {\bf On general weighted cumulative residual (past) extropy of extreme order statistics }} 
        
        \footnote{\noindent
                 {\bf } Corresponding author  E-mail: skchaudhary1994@gmail.com, santosh.chaudhary@cuj.ac.in
                 \\				{\bf $^2$} E-mail: : sarikul\_phd\_math@kgpian.iitkgp.ac.in\\
				{\bf $^3$} E-mail: nitin.gupta@maths.iitkgp.ac.in}
	\end{center}
	\begin{center}
		{\large {\bf Santosh Kumar Chaudhary$^1$},}
        {\large {\bf Sarikul Islam $^{2}$},} {\large {\bf Nitin Gupta$^{2}$}}\\ \vspace{0.5cm}
		{\large {\it $^{1}$ Department of Statistics, Central University of Jharkhand, Cheri-Manatu, Ranchi, Jharkhand   835222, India. }}\\
        {\large {\it $^{2,3}$ Department of Mathematics, Indian Institute of Technology Kharagpur, West Bengal 721302, India. }}\\
		\end{center}
	\vspace{.1in}
	\baselineskip 12truept
    \begin{center}
		{\bf \large Abstract}\\
	\end{center}
	  Weighted extropy has recently emerged as a flexible information measure for quantifying uncertainty, with particular relevance to order statistics. In this paper, we introduce and study a weighted cumulative analogue of extropy, extending the framework of weighted cumulative residual and cumulative past entropies to extreme order statistics. Specifically, we define the general weighted cumulative residual extropy (GWCREx) for the smallest order statistic and the general weighted cumulative past extropy (GWCPEx) for the largest order statistic, along with their dynamic versions. We show that these weighted measures—and their dynamic counterparts—uniquely characterize the underlying distribution. Moreover, we establish new characterization results for two widely used reliability models: the generalized Pareto distribution and the power distribution. The proposed framework provides a unified information‑theoretic tool for analysing extreme lifetimes in reliability engineering and survival analysis. \\   
        
	\vspace{1ex}
	\noindent \textbf{Keyword:} Cumulative Residual Extropy; Cumulative Past Extropy;  Extropy; Generalized Pareto Distribution; Order Statistics; Weighted Extropy. \\
\noindent \textbf{Mathematical Subject Classification:} 62B10, 62D05, 62G30, 94A17, 62E10.
	
\section{Introduction}
Shannon (1948) defined entropy for a discrete distribution \( P = (p_1, \dots, p_n) \) as 
\begin{align*}
H(P) = -\sum_{i=1}^{n} p_i \log (p_i),
\end{align*}
The concept of extropy was introduced by Lad, Sanfilippo, and Agrò (2015) as a dual to the well-known Shannon entropy, aimed at quantifying uncertainty through the complementary probabilities of events. Unlike entropy, which focuses on the surprise associated with observed outcomes, extropy measures the uncertainty arising from unobserved or unlikely outcomes. Mathematically, for a discrete distribution \( P = (p_1, \dots, p_n) \), extropy is defined as 
\begin{align*}
J(P) = -\sum_{i=1}^{n} (1 - p_i) \log (1 - p_i),
\end{align*}
highlighting its focus on the non-occurrence of events. Lad et al. (2018) further explored the duality of entropy and extropy, framing them within the context of the Kullback information complex, and demonstrated how both measures can jointly capture distributional uncertainty and information divergence. Building on this foundation, various generalizations have been proposed. Balakrishnan et al. (2020) introduced {weighted extropy} to account for the varying importance of outcomes, while Liu and Xiao (2021) and Balakrishnan et al.(2021) developed Rényi and Tsallis versions of extropy to enhance flexibility in non-extensive statistical mechanics and robust modeling.

Applications of extropy have flourished across multiple domains. In reliability and survival analysis, dynamic versions of extropy have been proposed to capture time-dependent uncertainty about remaining lifetime, paralleling the development of dynamic entropy. Buono et al.(2021) extended extropy to truncated distributions, proposing  weighted interval extropy to analyze censored data, particularly in biomedical and industrial settings. Extropy has also been effectively applied to ranked set sampling schemes. Gupta and Chaudhary (2023) proposed General weighted extropy (GWJ) as a tool to evaluate the performance of various ranked set sampling (RSS) schemes. Their work shows that GWJ outperforms entropy in distinguishing between sampling efficiencies and stochastic dominance. In the study of order statistics, Jose and Sathar (2019) analyzed the residual extropy of \( k \)-record values, contributing to the understanding of information content in the extremes.

Overall, extropy has emerged as a complementary yet distinct measure of uncertainty with growing theoretical support and a broad range of practical applications. Its strengths lie in capturing rare event uncertainty, tailoring importance through weights, and adapting to various data structures, including censored, dependent, and multivariate data. Despite its promise, further research is needed to explore robust estimation techniques, integration with machine learning models, and its role in complex system dynamics.

    Let \( X \) be an absolutely continuous non-negative random variable with probability density function (pdf) \( f(\cdot) \), cumulative distribution function (cdf) \( F(\cdot) \), and survival function (sf) \( \overline{F}(\cdot) \). Shannon (1948) defined differential entropy as
\begin{align}
H(X) = - \int_{0}^{\infty} f(x) \ln f(x) \, dx.
\end{align}
    Rao et al. (2004) defined cumulative residual entropy (CRE) as
\begin{align}
\xi H(X) = - \int_{0}^{\infty} \overline{F}(x) \ln \overline{F}(x) \, dx.
\end{align}
For more properties and applications of CRE, one may refer to Rao (2005). It is possible for two distinct distributions to share the same CRE, indicating that CRE alone does not uniquely identify a distribution. Baratpour (2010) investigated various properties of CRE and developed necessary and sufficient conditions under which the CRE of the first-order statistic uniquely determines the underlying distribution, up to the location parameter. The Weibull distribution family, in particular, was characterized using the ratio of the CRE of the first-order statistic to its expected value. Additionally, Baratpour (2010) established that the CRE of a component lifetime, \(X\), in an \(n\)-component series system is bounded below by the CRE of the system lifetime \(X_{1:n}\). Hashempour and Doostparast (2020) addressed the problem of identifying the parent distribution using the CRE of sequential order statistics within the framework of a conditional proportional hazard rates model. They demonstrated that if the CREs of the initial sequential order statistics are equal, the parent distribution is uniquely determined. Furthermore, they characterized the Weibull distribution by examining the ratio of the CRE of the first sequential order statistic to its expected value. The study also explored characterizations involving dynamic cumulative residual entropy and provided bounds for the CRE of the residual lifetime associated with sequential order statistics.

Later, Di Crescenzo and Longobardi (2009) introduced a dual measure based on the cdf \( F(x) \), called the cumulative past entropy (CPE), which is analogous to CRE, as follows
\begin{align*}
\bar{\xi} H(X) = - \int_{0}^{\infty} F(x) \ln F(x) \, dx.
\end{align*}
Lad et al. (2015) defined differential extropy as
\begin{align*}
J(X) = -\frac{1}{2}  \int_{0}^{\infty} f^2_X(x) \, dx
\end{align*}
Jahansahi et al (2020) defined cumulative residual extropy (CRJ) as
\begin{align*}
\xi J(X) = -\frac{1}{2} \int_{0}^{\infty} \overline{F}(x)^2 \, dx
\end{align*}
The cumulative past extropy (CPJ) can be defined as
\begin{align*}
\overline{\xi} J(X) = -\frac{1}{2} \int_{0}^{\infty} F(x)^2 \, dx
\end{align*}
Kundu (2023) considered a cumulative form of extropy, drawing inspiration from the concepts of cumulative residual and past entropy. This new measure was explored within the framework of extreme order statistics. The study also examined dynamic forms of cumulative residual (and past) extropy related to the minimum and maximum order statistics. It was demonstrated that these newly proposed metrics, including their dynamic counterparts, can uniquely identify the underlying distribution. Additionally, the research provided characterizations for the generalized Pareto and power distributions, which are frequently applied in reliability analysis. Kattumannil et al. (2024) investigated the link between cumulative residual entropy and the Gini mean difference (GMD). They also explored relationships between extropy and GMD, as well as between truncated GMD and dynamic forms of cumulative past extropy. Furthermore, their study demonstrated that various entropy and extropy measures discussed in the work can be expressed within the framework of probability-weighted moments. This formulation provides a useful basis for developing estimators for these measures.

The Negative cumulative extropy (NCJ) for a non-negative, continuous random variable \( X \) is defined by Tahmasebi and Toomaj (2022) as
\begin{align*}
NCJ(X) = \frac{1}{2} \int_0^\infty \left(1 - F_X^2(x)\right) \, dx
\end{align*}
The general weighted negative cumulative extropy (GWNCJ) for a  non-negative, absolutely continuous random variable \( X \) with a weight function \( w(x) > 0 \) is defined by Gupta and Chaudhary (2024) as
\begin{align*}
\overline{\eta}_w J(X) = \frac{1}{2} \int_0^\infty w(x)\left(1 - F_X^2(x)\right) \, dx.
\end{align*}
Gupta and Chaudhary (2023a) defined GWJ as 
\begin{align*}
J_w(X) &= -\frac{1}{2}  \int_{0}^{\infty} w(x) f^2_X(x) \, dx. \end{align*}

Gupta and Chaudhary (2023b) defined general weighted cumulative past extropy (GWCPEx) as
\begin{align*}
\overline{\xi}_J^{w}(X) &= -\frac{1}{2} \int_{0}^{\infty} w(x) F^2(x) \, dx. 
\end{align*}
Gupta and Chaudhary (2023c) defined general weighted cumulative residual extropy (GWCREx)  as
\begin{align*}
\xi_J^w(X) &= -\frac{1}{2} \int_{0}^{\infty} w(x) \overline{F}^2(x) \, dx.
\end{align*}


The weight function gives practitioners the flexibility to tailor the uncertainty measure to specific economic or risk criteria. For instance, in warranty analysis, later failures are often more costly; using \(w(x)=x\) or \(w(x)=x^2\) naturally emphasises those periods. The characterisation results provide simple, moment-type checks for commonly used lifetime models (exponential, Pareto, power), without requiring full maximum likelihood estimation. Thus, the proposed measures bridge information theory and applied reliability engineering. Kattumannil et al. (2022) proposed a generalized form of cumulative residual entropy and explored its theoretical properties. They demonstrated that several well-known entropy measures—such as cumulative residual entropy, weighted cumulative residual entropy, and cumulative residual Tsallis entropy—can be viewed as specific instances of this broader framework. In addition, they introduced a generalized cumulative entropy measure encompassing cumulative entropy, its weighted version, and cumulative Tsallis entropy as special cases. The authors also employed a generating function approach to derive various entropy measures. Moreover, they presented cumulative and residual variants of the Sharma–Taneja–Mittal entropy, showing how these too arise as special cases of their generalized measure. Finally, their study revealed connections between the newly defined entropy measures and corresponding extropy measures. For additional research on extropy and the measures GWJ, GWCJ, GWRJ, and GWNCJ, refer to the works of Chaudhary and Gupta (2025, 2024a, 2024b, 2023a, 2023b, 2023c) and Sahu and Gupta (2024). 
While weighted extropy (Balakrishnan et al., 2020) and cumulative residual extropy (Jahansahi et al., 2020) have been studied separately, this work is the first to combine both weighting and cumulation for extreme order statistics. Unlike previous characterizations that rely on entropy or unweighted extropy (Kundu, 2023), we prove that the weighted cumulative versions uniquely determine the parent distribution up to location (or location‑scale) under a mild Müntz–Szász condition. Furthermore, the dynamic weighted measures lead to new, simple criteria for identifying the generalized Pareto and power distributions – models widely used in reliability, hydrology, and extreme value theory. The weight function allows practitioners to emphasise specific regions of the support (e.g., early failures or high‑cost lifetimes), which is directly relevant in risk assessment, warranty analysis, and preventive maintenance scheduling. The remainder of this paper is organized as follows. Section 2 introduces the general weighted cumulative residual extropy (GWCREx) of the smallest order statistic and provides explicit expressions for several distributions. Section 3 develops the dynamic version (GWDCREx) and establishes monotonicity properties, bounds, and a characterization of the generalized Pareto distribution via a constant‑derivative condition. Section 4 presents the analogous results for the weighted cumulative past extropy (GWCPEx) of the largest order statistic. Section 5 extends these to dynamic past extropy and gives a characterization of the power distribution. A real‑data illustration in Section 6 demonstrates how the proposed measures can be estimated and interpreted in practice. Section 7 concludes this paper.

\section{Results on GWCREx of Smallest Order Statistic}
The GWCREx of the smallest order statistic \(X_{1:n}\), based on an independent and identically distributed (i.i.d.) sample \(X_1, X_2, \ldots, X_n\) drawn from a distribution \(F\), can be written as
\begin{align}
    \xi_J^w(X_{1:n}) 
    &= -\frac{1}{2} \int_{0}^{\infty} w(x) \left[\overline{F}_{1:n}(x)\right]^2 dx \nonumber\\
    &= -\frac{1}{2} \int_{0}^{\infty} w(x) \left[\overline{F}(x)\right]^{2n} dx \nonumber\\
    &= -\frac{1}{2} \int_{0}^{1} \frac{w(F^{-1}(1 - u)) \cdot u^{2n}}{f(F^{-1}(1 - u))} du. \label{GWCRJ}
\end{align}
Let us discuss some examples to observe the behaviour of the GWCREx of \(X_{1:n}\) for different distributions.
\begin{example}
    Let $X$ be a random variable following a Uniform [a, b] distribution with \( a < b \).
    If \( X_{1:n} \) is the minimum of an i.i.d. sample of size \( n \) drawn from \(X\), then
\[
\xi_J^w(X_{1:n}) = -\frac{1}{2} \int_a^b w(x) \left( \frac{b - x}{b - a} \right)^{2n} dx= \frac{1}{2}(b - a) \int_0^1 w(b - u(b - a)) \cdot u^{2n} \, du.
\]
In particular, if
\begin{enumerate}[(i)]
    \item  \( w(x) = 1 \), then
\[
\xi_J^w(X_{1:n}) = \frac{1}{2}(b - a) \cdot \frac{1}{2n + 1} = \frac{b - a}{2(2n + 1)},
\]
\item \( w(x) = x \), then
\[
\xi_J^w(X_{1:n}) = \frac{1}{2}(b - a) \left[ \frac{b}{2n + 1} - \frac{b - a}{2n + 2} \right],
\]
\item \( w(x) = \log x \), then
\[
\xi_J^w(X_{1:n}) = \frac{1}{2}(b - a) \int_0^1 \log(b - u(b - a)) \cdot u^{2n} \, du,
\]
\item 
  \( w(x) = x^m \), with \( m \in \mathbb{N} \), then
\[
\xi_J^w(X_{1:n}) = \frac{1}{2}(b - a) \sum_{k=0}^{m} \binom{m}{k} a^k b^{m - k} \cdot \frac{\Gamma(2n + k + 1)\Gamma(m - k + 1)}{\Gamma(2n + m + 2)}.
\]
\end{enumerate}
\end{example}
\begin{example}
Let \( X \) be random variable with cdf \( F(x) = 1 - (1 - a x)^b \), for \( x \in (0, 1/a) \), where \( a, b > 0 \). Then
\begin{align*}
\xi_J^w(X_{1:n}) &= -\frac{1}{2} \int_0^{1/a} w(x) \cdot \left( \overline{F}(x) \right)^{2n} dx \\
&= -\frac{1}{2} \int_0^{1/a} w(x) (1 - a x)^{2nb} dx\\
&= \frac{1}{2a} \int_0^1 w\left( \frac{1 - u}{a} \right) u^{2nb} \, du.
\end{align*}
In particular, if
\begin{enumerate}[(i)] 
    \item  \( w(x) = 1 \), then
    \[
    \xi_J^w(X_{1:n}) = \frac{1}{2a} \cdot \int_0^1 u^{2nb} \, du = \frac{1}{2a(2nb + 1)},
    \]
\item  \( w(x) = x \), then
    \[
    \xi_J^w(X_{1:n}) = \frac{1}{2a} \int_0^1 \left( \frac{1 - u}{a} \right) u^{2nb} du = \frac{1}{2a^2(2nb + 1)(2nb + 2)},
    \]
\item  \( w(x) = \log x \), then
    \[
    \xi_J^w(X_{1:n}) = \frac{1}{2a} \int_0^1 \log\left( \frac{1 - u}{a} \right) u^{2nb} \, du,
    \]
\item \( w(x) = x^m \), where \( m \in \mathbb{N} \), then
\[
\xi_J^w(X_{1:n}) = \frac{1}{2a^{m+1}} \int_0^1 (1 - u)^m u^{2nb} \, du = \frac{1}{2a^{m+1}} \cdot \frac{\Gamma(2nb + 1)\Gamma(m + 1)}{\Gamma(2nb + m + 2)}.
\]
\end{enumerate}
\end{example}

\begin{example}
Let \( X \sim \mathrm{Weibull}(k,h) \) with cdf 
\( F(x)=1-\exp(-k x^{h}) \), for \( x>0 \), where \( k,h>0 \). Then
\[
\xi_J^w(X_{1:n}) = -\frac{1}{2} \int_0^\infty w(x) \cdot e^{-2nk x^h} \, dx.=-\frac{1}{2h} \int_0^\infty w(u^{1/h}) \cdot u^{\frac{1}{h} - 1} \cdot e^{-2nk u} \, du.
\]
Now if  \begin{enumerate}[(i)]
    \item \( w(x) = 1 \), then
    \[
    \xi_J^w(X_{1:n}) = -\frac{1}{2h} \cdot \Gamma\left( \frac{1}{h} \right) \cdot (2nk)^{-\frac{1}{h}},
    \]
 \item  \( w(x) = x \), then
    \[
    \xi_J^w(X_{1:n}) = -\frac{1}{2h} \cdot \Gamma\left( \frac{2}{h} \right) \cdot (2nk)^{-\frac{2}{h}},
    \]
\item  \( w(x) = x^m \), where \( m \in \mathbb{N}\), then
\[
\xi_J^w(X_{1:n}) = -\frac{1}{2} \int_0^\infty x^m e^{-2nk x^h} \, dx 
= -\frac{1}{2h} \cdot \Gamma\left( \frac{m+1}{h} \right) \cdot (2nk)^{-\frac{m+1}{h}}.
\]
\end{enumerate} 
\end{example}
\begin{example}
Let \( X \) follow a Folded Cramér distribution with cdf \( F(x) = 1 - \frac{1}{1 + h x} \), for \( x > 0 \), and \( h > 0 \). Then the sf of \(X\) is written as \( \overline{F}(x) = \frac{1}{1 + h x} \). 
For the minimum order statistic \(X_{1:n}\) and a weight function \(w(\cdot)\), the GWCRJ of \(X_{1:n}\) is given by
\[
\xi_J^w(X_{1:n}) = -\frac{1}{2} \int_0^\infty w(x) \cdot \left( \frac{1}{1 + h x} \right)^{2n} dx.
\]
Using the substitution \( u = 1 + h x \), we get
\[
\xi_J^w(X_{1:n}) = -\frac{1}{2h} \int_1^\infty w\left( \frac{u - 1}{h} \right) \cdot u^{-2n} \, du.
\]
Now if \begin{enumerate}[(i)]
    \item \( w(x) = 1 \), then
    \[
    \xi_J^w(X_{1:n}) = -\frac{1}{2h(2n - 1)},
    \]
\item  \( w(x) = x \), then
    \[
    \xi_J^w(X_{1:n}) = -\frac{1}{2h^2(2n - 2)(2n - 1)},
    \]
\item  \( w(x) = x^m \), with \( m \in \mathbb{N} \), then
\[
\xi_J^w(X_{1:n}) = -\frac{1}{2h^{m+1}} \int_1^\infty (u - 1)^m \cdot u^{-2n} \, du= -\frac{1}{2h^{m+1}} \cdot B(m + 1, 2n - m - 1),
\]
where \( B(a, b) \) is the Beta function, provided \( 2n > m + 1 \).
\end{enumerate} 
\end{example}
\begin{example}
Let \( X \sim \text{Pareto}(k, h) \) with sf \( \overline{F}(x) = \left( \frac{k}{x + k} \right)^h \), where \( x > 0 \), \( k > 0 \), and \( h > 1 \). Then the GWCRJ measure of \(X_{1:n}\) is given by 
\[
\xi_J^w(X_{1:n}) = -\frac{1}{2} \int_0^\infty w(x) \cdot \left( \frac{k}{x + k} \right)^{2nh} dx.
\]
Now if \begin{enumerate}[(i)]
    \item \( w(x) = 1 \), then
    \[
    \xi_J^w(X_{1:n}) = -\frac{k}{2(2nh - 1)},
    \]
\item  \( w(x) = x \), then
    \[
    \xi_J^w(X_{1:n}) = -\frac{k^2}{2(2nh - 2)(2nh - 1)},
    \]
\item  \( w(x) = \log x \), then
    \[
    \xi_J^w(X_{1:n}) = -\frac{1}{2} k^{2nh} \int_k^\infty \log(u - k) \cdot u^{-2nh} \, du,
    \]
 \item  \( w(x) = x^m \), with \( m \in \mathbb{N} \), then:
\[
\xi_J^w(X_{1:n}) = -\frac{k^{m + 1}}{2} \cdot B(m + 1, 2nh - m - 1),
\]
where \( B(a, b) \) is the Beta function, valid for \( 2nh > m + 1 \).
\end{enumerate} 
\end{example}

\begin{proposition} \label{proposition1}
Let \( X_1, X_2, \ldots, X_n \) be an i.i.d. sample from a non-negative continuous distribution with sf \( \overline{F}(x) \) with  the minimum order statistic \( X_{1:n} = \min\{X_1, \ldots, X_n\} \). Then the weighted cumulative residual extropy of \( X_{1:n} \) is {increasing} in \( n \), provided that \( w(x) \geq 0 \) for all \( x \geq 0 \). 
\end{proposition}
\textbf{Proof.} Since \( \overline{F}(x) \in (0, 1) \) for \( x > 0 \), the function \( \overline{F}(x)^{2n} \) decreases as \( n \) increases. Hence, \( \xi_J^w(X_{1:n}) \) increases with \( n \) for non-negative weight functions \( w(x) \). \par 
An obvious interpretation of Proposition~\ref{proposition1} is that the weighted cumulative residual extropy \(\xi_J^w\) of an \(n+i\) component series system, where \(n,\,i\in\mathbb{N}\), is always greater than that of an \(n\) component series system.
\begin{proposition} \label{proposition2}
Let an i.i.d. sample \( X_1, X_2, \dots, X_n \) drawn from a non-negative absolutely continuous random variable \( X \) with finite mean \( \mu = \mathbb{E}[X] \). Let \( w(x) \geq 0 \) be a weight function on \( [0, \infty) \). Then
\begin{itemize}
  \item[(i)] \( \xi_J^w(X_{1:n})  \geq- \dfrac{\mu_w}{2} \), where \( \mu_w := \int_0^\infty w(x) \overline{F}(x) \, dx \).
  \item[(ii)] \( \xi_J^w(X_{1:n}) \geq \xi_J^w(X) \).
\end{itemize}

\end{proposition}
\noindent \textbf{Proof.}   (i) Since the sample is i.i.d., the sf of the minimum order statistic is:
\[
\overline{F}_{1:n}(x) = \mathbb{P}(X_{1:n} > x) = [\overline{F}(x)]^n.
\]
Hence, the weighted residual extropy of \( X_{1:n} \) is:
\[
\xi_J^w(X_{1:n}) = -\frac{1}{2} \int_0^\infty w(x) \left[\overline{F}(x)\right]^{2n} dx.
\]
Since \( 0 \leq \overline{F}(x) \leq 1 \), we have:
\[
\overline{F}^{2n}(x) \leq \overline{F}(x),
\]
so
\[
\xi_J^w(X_{1:n}) = -\frac{1}{2} \int_0^\infty w(x) \overline{F}^{2n}(x) dx \geq -\frac{1}{2} \int_0^\infty w(x) \overline{F}(x) dx.
\]
Define:
\[
\mu_w := \int_0^\infty w(x) \overline{F}(x) dx.
\]
So we conclude that
\[
\xi_J^w(X_{1:n}) \geq- \frac{\mu_w}{2}.
\]
(ii) Again, using \( \overline{F}^{2n}(x) \leq \overline{F}^2(x) \) for all \(n\in \mathbb{N}\), we have:
\[
\xi_J^w(X_{1:n}) = -\frac{1}{2} \int_0^\infty w(x) \overline{F}^{2n}(x) dx \geq -\frac{1}{2} \int_0^\infty w(x) \overline{F}^2(x) dx = \xi_J^w(X).
\]

\begin{lemma}\label{lemma01}
(Müntz--Szász Theorem) Let \( \{n_j\}_{j=1}^\infty \) be a strictly increasing sequence of positive integers. Then, the set of monomials \( \{x^{n_j}\}_{j=1}^\infty \) is complete in \( L^2(0,1) \) if and only if
\[
\sum_{j=1}^{\infty} \frac{1}{n_j} = \infty.
\]
\end{lemma}
\noindent In the sequel, we assume that \( \{n_j\}_{j \geq 1} \) is a strictly increasing sequence of positive integers satisfying the condition $\sum_{j=1}^{\infty} \frac{1}{n_j} = \infty.$

The following theorem shows that the parent distribution can be uniquely characterized (up to location) in terms of the GWCREx of the smallest order statistics.

\begin{theorem}
Let \(X\) and \(Y\) be nonnegative absolutely continuous random variables with cdfs \(F\) and \(G\), and sfs \(\overline F\) and \(\overline G\), respectively. Let \(\{n_j\}_{j=1}^{\infty}\) be a strictly increasing sequence of positive integers such that
\[
\sum_{j=1}^{\infty}\frac{1}{n_j}=\infty.
\]
Let \(w:(0,\infty)\to(0,\infty)\) be a differentiable function satisfying \(w(\theta+x)=w(x)\) for some \(\theta\in\mathbb{R}\) and all \(x>0\). Then \(F\) and \(G\) belong to the same location family, i.e.,
\(
F(x)=G(x-\theta),
\) 
if and only if
\[
\xi_J^w(X_{1:n}) = \xi_J^w(Y_{1:n}) \quad \text{for all } n = n_j.
\]
\end{theorem}
\noindent \textbf{Proof.}  \textbf{(Necessity)}: Assume \( F(x) = G(x - \theta) \). Then,
\[
\overline{F}(x) = \overline{G}(x - \theta) \quad \Rightarrow \quad \overline{F}^{2n}(x) = \overline{G}^{2n}(x - \theta).
\]
Substituting into the definition of \( \xi_J^w(X_{1:n}) \),
\begin{align*}
\xi_J^w(X_{1:n}) &= -\frac{1}{2} \int_0^\infty w(x) \overline{F}^{2n}(x) \, dx \\
&= -\frac{1}{2} \int_0^\infty w(x) \overline{G}^{2n}(x - \theta) \, dx.
\end{align*}
Letting \( u = x - \theta \), so that \( x = u + \theta \) and \( dx = du \), we get
\[
\xi_J^w(X_{1:n}) = -\frac{1}{2} \int_{-\theta}^{\infty} w(u + \theta) \overline{G}^{2n}(u) \, du.
\]
Since $w(\theta+x)=w(x)$ on \(x\in (0, \infty) \) then this implies
\[
\xi_J^w(X_{1:n}) = \xi_J^w(Y_{1:n}).
\]
\textbf{(Sufficiency)}: Assume \( \xi_J^w(X_{1:n}) = \xi_J^w(Y_{1:n}) \) for all \( n = n_j \), and
\[
\sum_{j=1}^{\infty} \frac{1}{n_j} = \infty.
\]
Then,
\[
\int_0^\infty w(x) \left[ \overline{F}^{2n_j}(x) - \overline{G}^{2n_j}(x) \right] dx = 0, \quad \forall \;j\in \mathbb{N}.
\]
Make a change of variable \( u = \overline{F}(x) \). Assuming invertibility and differentiability of \( \overline{F} \), we have
\[
\int_0^1 u^{2n_j} \left[ \frac{w(F^{-1}(1-u))}{f(F^{-1}(1-u))} - \frac{w(G^{-1}(1-u))}{g(G^{-1}(1-u))} \right] du = 0.
\]
Define:
\[
h(u) := \frac{w(F^{-1}(1-u))}{f(F^{-1}(1-u))} - \frac{w(G^{-1}(1-u))}{g(G^{-1}(1-u))}.
\]
Then,
\[
\int_0^1 u^{2n_j} h(u) \, du = 0, \quad \forall j\in \mathbb{N}.
\]
From Lemma \ref{lemma01}, the set \( \{u^{2n_j}\} \) is complete in \( L^2(0,1) \) if \( \sum \frac{1}{n_j} = \infty \). Hence, \( h(u) = 0 \) almost everywhere. This implies
\begin{equation}
  \frac{w(F^{-1}(1-u))}{f(F^{-1}(1-u))} = \frac{w(G^{-1}(1-u))}{g(G^{-1}(1-u))}, \quad \text{a.e. on } (0,1).\label{identity}  
\end{equation}
The functional identity in Eq.~\eqref{identity} implies a location shift between \( F \) and \( G \). Hence, there exists \( \theta \in \mathbb{R} \) such that
\[
F(x) = G(x - \theta), \quad \text{or equivalently, } F^{-1}(u) = G^{-1}(u) + \theta.
\]

\section{Results on General Weighted Dynamic Cumulative Residual Extropy (GWDCREx) of Smallest Order Statistic}

In addition to dynamic entropy measures,  dynamic extropy-based measures have been proposed to capture uncertainty from a complementary perspective. Extropy, often interpreted as a dual of entropy, quantifies the ``missing'' or ``complementary'' information and has been extended to dynamic settings as well.

To incorporate the impact of time and the relative importance of different lifetime values,  {weighted versions} of dynamic extropy measures have been introduced. These allow the analyst to emphasize certain parts of the residual lifetime distribution based on a weight function.

Let \( X \) be a nonnegative random variable representing the lifetime of a component or system. Given survival up to time \( t \), the residual lifetime is defined as:
\[
X_t = (X - t \mid X > t).
\]
Let \( \bar{F}(x) = 1 - F(x) \) denote the sf of \( X \), and let \( w(x) \) be a nonnegative, integrable  {weight function} defined for \( x \ge t \).

The  general weighted dynamic cumulative residual extropy (GWDCREx) of \( X \) at time \( t \) is defined as
\[
\mathcal{E}_J^w(X; t) = -\frac{1}{2} \int_t^{\infty} w(x) \left( \frac{\bar{F}(x)}{\bar{F}(t)} \right)^2 dx.
\]

Analogously, the GWDCREx of the  {minimum order statistic} \( X_{1:n} \), representing the first failure time among \( n \) i.i.d. components of a coherent system with lifetime distribution \( F \), is given by
\[
\mathcal{E}_J^w(X_{1:n}; t) = -\frac{1}{2} \int_t^{\infty} w(x) \left( \frac{\bar{F}(x)}{\bar{F}(t)} \right)^{2n} dx.
\]

These measures provide a flexible framework for quantifying uncertainty in systems that have aged to time \( t \), with the added ability to prioritize certain lifetime intervals via the weighting function \( w(x) \). Such formulations are particularly useful in reliability analysis, where the cost or risk associated with failure may vary over time.

 Note that when \( w(x) = 1 \), the GWDCREx reduces to the standard Dynamic Cumulative Residual Extropy (DCREx). Differentiating \( \mathcal{E}_J^w(X_{1:n}; t) \) with respect to \( t \), we apply Leibniz's rule and obtain  
\begin{equation}
 \frac{d}{dt} \mathcal{E}_J^w(X_{1:n}; t) = 2n \cdot k_F(t) \cdot \mathcal{E}_J^w(X_{1:n}; t) - \frac{1}{2} w(t),\label{equationdiff}   
\end{equation}
where \( k_F(t) = \frac{f(t)}{\bar{F}(t)} \) is the hazard rate function of \( X \). Therefore, the monotonicity of \( \mathcal{E}_J^w(X_{1:n}; t) \) depends on the sign of the derivative and is given as follows:
\begin{enumerate}[(i)]
    \item \( \mathcal{E}_J^w(X_{1:n}; t) \) is \textbf{increasing} in \( t \) if and only if
    \[
    \mathcal{E}_J^w(X_{1:n}; t) > \frac{w(t)}{4n k_F(t)},
    \]
    \item \( \mathcal{E}_J^w(X_{1:n}; t) \) is \textbf{decreasing} in \( t \) if and only if
    \[
    \mathcal{E}_J^w(X_{1:n}; t) < \frac{w(t)}{4n k_F(t)}.
    \]
\end{enumerate}

Here, we provide lower bounds and monotonicity properties for the GWDCREx for the minimum order statistic \( X_{1:n} \). Let \( w(x) \) be a non-negative weight function and \( d_F(t) = \mathbb{E}[X_t] \) be the mean residual life (mrl) function of \( X \), where \( X_t = (X - t \mid X > t) \).

\begin{proposition}[Dynamic version of Proposition~\ref{proposition2}]
Let \( \mathcal{E}_J^w(X_{1:n}; t) \) denote the GWDCREx of \( X_{1:n} \) at time \( t \). Then the following results hold:
\begin{itemize}
    \item[(i)] 
    \[
    \mathcal{E}_J^w(X_{1:n}; t) \geq \frac{w(t) \cdot d_F(t)}{2};
    \]
    \item[(ii)] 
    \( \mathcal{E}_J^w(X_{1:n}; t) \) is increasing in \( n \), and
    \[
    \mathcal{E}_J^w(X_{1:n}; t) \geq \mathcal{E}_J^w(X; t).
    \]
\end{itemize}
\end{proposition}

Let \( X \) follow a finite range distribution and suppose the weight function is constant, i.e., \( w(x) = w_0 > 0 \). Then
\[
\mathcal{E}_J^w(X_{1:n}; t) = w_0 \left( \frac{1 + b}{1 + 2nb} \right) \cdot \frac{d_F(t)}{2} \leq \frac{w_0 d_F(t)}{2},
\]
which confirms part (i) of the proposition.

Moreover, the difference between the GWDCREx of the minimum order statistic and the baseline variable is
\[
\mathcal{E}_J^w(X_{1:n}; t) - \mathcal{E}_J^w(X; t) = w_0 \cdot \frac{2b(n - 1)(1 - at)}{2a(1 + 2nb)(1 + 2b)} \geq 0, \quad \forall n \geq 1,
\]
which verifies part (ii). The inequality in (i) gives a simple and interpretable lower bound of GWDCREx for a series system that incorporates both the weighting scheme \( w(t) \) and the residual life \( d_F(t) \). As the number of components \( n \) increases, the series system's overall uncertainty, as captured by GWDCREx, increases.  When \( w(x) \) is increasing in \( x \), higher weights are assigned to longer residual lifetimes, amplifying the extropy effect. For details on stochastic orders, one may refer to Shaked and Shanthikumar (2007). 

\begin{definition}[WDCREx stochastic order]
Let \( X \) and \( Y \) be two non-negative random variables, and let \( w(x) \) be a non-negative weight function defined on \( [0, \infty) \). Then \( X \) is said to be less than or equal to \( Y \) in the WDCREx order, denoted by
\[
X \leq_{\mathrm{WDCREx}} Y,
\]
if and only if
\[
\mathcal{E}_J^w(X; t) \geq \mathcal{E}_J^w(Y; t), \quad \text{for all } t \geq 0,
\]
\end{definition}
\begin{example}

Let \( X_i\) follow  the \(\text{Pareto}(k, h_i) \), with sf
\[
\bar{F}_i(x) = \left( \frac{k}{x + k} \right)^{h_i}, \quad x > 0,\; k > 0,\; h_i > 1,\; i = 1, 2.
\]
Let the weight function be \( w(x) = x \). Then the GWDCREx of \( X_i \)'s are given by
\[
\mathcal{E}_J^w(X_i; t) =-\frac{1}{2} \int_t^{\infty} x \left( \frac{k + t}{k + x} \right)^{2h_i} dx.
\]
Since the exponent \( 2h_i \) increases with \( h_i \), the integrand becomes smaller for larger \( h_i \). Therefore, for \( h_1 > h_2 \), we have
\[
\mathcal{E}_J^w(X_1; t) < \mathcal{E}_J^w(X_2; t), \quad \forall t \geq 0,
\]
which implies the ordering
\[
X_2 \leq_{\mathrm{WDCREx}} X_1.
\]
Thus, \( X_1 \) is less uncertain than \( X_2 \) in the sense of weighted dynamic cumulative residual extropy under the chosen weight function.
\end{example}

\begin{theorem}
Let \( X_1 \) and \( X_2 \) be two non-negative random variables, and let \( Y_i = aX_i + b \), for \( i = 1, 2 \), with \( a > 0 \) and \( b \geq 0 \). Suppose \( w(x) \) is a non-negative weight function, and define the transformed weight function as \( w^*(x) = w\left( \frac{x - b}{a} \right) \). Then
\[
X_1 \leq_{\mathrm{WDCREx}} X_2 \quad \Rightarrow \quad Y_1 \leq_{\mathrm{WDCREx}} Y_2,
\]
that is,
\[
\mathcal{E}_J^{w}(X_1; t) \geq \mathcal{E}_J^{w}(X_2; t) \quad \Rightarrow \quad \mathcal{E}_J^{w^*}(Y_1; t) \geq \mathcal{E}_J^{w^*}(Y_2; t), \quad \forall\, t \geq 0.
\]
\end{theorem}

\begin{theorem}
Let \( X_1 \) and \( X_2 \) be two non-negative random variables. Define the affine transformations
\[
Y_i = a_i X_i + b_i, \quad i = 1, 2,
\]
with \( a_1 \ge a_2 > 0 \) and \( b_1 \ge b_2 > 0 \). Let \( w(x) \) be a non-negative weight function, and define the corresponding transformed weight functions as
\[
w_i^*(x) = w\left( \frac{x - b_i}{a_i} \right), \quad i = 1, 2.
\]
If
\[
X_1 \leq_{\mathrm{WDCREx}} X_2,
\]
and either \( \mathcal{E}_J^w(X_1; t) \), or \( \mathcal{E}_J^w(X_2; t) \) is decreasing in \( t \), then
\[
Y_1 \leq_{\mathrm{WDCREx}} Y_2.
\]
\end{theorem}

\begin{theorem} \label{theorem4}
Let \( X \) and \( Y \) be two absolutely continuous non-negative random variables with hazard rate functions \( k_F(t) \) and \( k_G(t) \), respectively. If
\[
X \leq_{\mathrm{hr}} Y, \quad \text{i.e.,} \quad k_F(t) \geq k_G(t), \quad \forall\, t \geq 0,
\]
then for all \( n \in \mathbb{N} \),
\[
X_{1:n} \leq_{\mathrm{DCREx}} Y_{1:n}, \quad \text{that is,} \quad \mathcal{E}_J(X_{1:n}; t) \geq \mathcal{E}_J(Y_{1:n}; t), \quad \forall\, t \geq 0.
\]
\end{theorem}
An obvious interpretation of Theorem~\ref{theorem4} is that the higher hazard rate of a series system implies the greater dynamic cumulative residual extropy of the system.
\begin{example}
Let \( X_{h}\) follows \(\mathrm{Weibull}(\lambda, h) \), with \( h > 0 \), and let \( w(x) \) be a non-negative, non-decreasing weight function. The sf of \( X_h \) is given by
\[
\bar{F}_h(x) = \exp(-\lambda x^h), \quad x \geq 0,
\]
and the corresponding hazard rate function is
\[
k_F(x) = \lambda h x^{h-1}.
\]
It follows that for \( h_1 > h_2 \),
\[
k_{F_{h_1}}(t) > k_{F_{h_2}}(t), \quad \forall t \geq 0,
\]
which implies the hazard rate ordering \( X_{h_1} \le_{\mathrm{hr}} X_{h_2} \), and hence Theorem~\ref{theorem4} implies that
\[
X_{1:n, h_1} \leq_{\mathrm{DCREx}} X_{1:n, h_2}.
\]

Now, consider the GWDCREx of the minimum order statistic
\[
\mathcal{E}_J^w(X_{1:n, h}; t) =- \frac{1}{2} \int_t^{\infty} w(x) \left( \frac{\bar{F}_h(x)}{\bar{F}_h(t)} \right)^{2n} dx.
\]
Since \( \bar{F}_{h_1}(x)/\bar{F}_{h_1}(t) < \bar{F}_{h_2}(x)/\bar{F}_{h_2}(t) \) for all \( x > t \), and \( w(x) \) is non-decreasing, it follows that
\[
\mathcal{E}_J^w(X_{1:n, h_1}; t) < \mathcal{E}_J^w(X_{1:n, h_2}; t),
\]
and hence,
\[
X_{1:n, h_1} \leq_{\mathrm{WDCREx}} X_{1:n, h_2}.
\]
   
\end{example}

\begin{theorem}
 Let \(X_{k:n}\) denote the \(k\)-th order statistic from an i.i.d. sample of size \(n\) drawn from an absolutely continuous distribution with cdf \(F\). Let \(w(\cdot)\) be a non-negative, non-decreasing weight function. Then, for all \(t \ge 0\), the following inequalities hold:
\begin{enumerate}[(i)]
    \item \( \mathcal{E}_J^w(X_{k:n}; t) \geq \mathcal{E}_J^w(X_{k+1:n}; t) \)
    \item \( \mathcal{E}_J^w(X_{k:n}; t) \geq \mathcal{E}_J^w(X_{k:n-1}; t) \)
    \item \( \mathcal{E}_J^w(X_{k:n}; t) \geq \mathcal{E}_J^w(X_{k+1:n+1}; t) \)
\end{enumerate}
\end{theorem}

These inequalities demonstrate that under a reasonable weighting scheme, the dynamic uncertainty associated with earlier order statistics is greater than that of later or more redundant system structures.
\begin{theorem}
Let \(X\) and \(Y\) be nonnegative, absolutely continuous random variables with cdf's \(F\) and \(G\), respectively. Let \( w(x) \) be a non-negative, measurable weight function such that the general weighted dynamic cumulative residual extropy (GWDCREx) is finite. Define the GWDCREx of \(X_{1:n}\) of a sample of size \( n \) as
\[
\mathcal{E}_J^w(X_{1:n}; t) = -\frac{1}{2} \int_t^\infty w(x) \left( \frac{\bar{F}(x)}{\bar{F}(t)} \right)^{2n} dx.
\]
Then, \( F \) and \( G \) belong to the same location family, i.e., there exists \( \theta \in \mathbb{R} \) such that
\[
F(x) = G(x - \theta), \quad \text{for all } x \in \mathbb{R},
\]
if and only if
\[
\mathcal{E}_J^w(X_{1:n_j}; t) = \mathcal{E}_J^w(Y_{1:n_j}; t), \quad \forall j \ge 1,
\]
for a sequence \( \{n_j\}_{j=1}^{\infty} \subset \mathbb{N} \) such that
\[
\sum_{j=1}^{\infty} \frac{1}{n_j} = \infty.
\]

\end{theorem}

\begin{theorem}
  
Let \( X \) be a non-negative absolutely continuous random variable with mrl function \( d_F(t) \). Suppose the general weighted dynamic cumulative residual extropy (GWDCREx) of \( X_{1:n} \) with weight function \( w(x) \) satisfies
\begin{equation} \label{eq:wdcrex_mrl}
\mathcal{E}_J^w(X_{1:n}; t) = c \, d_F(t), \quad \text{for all } t \geq 0,
\end{equation}
where \( c \) is a constant and
\[
\mathcal{E}_J^w(X_{1:n}; t) := -\frac{1}{2} \int_t^\infty w(x) \left( \frac{\bar{F}(x)}{\bar{F}(t)} \right)^{2n} dx.
\]
Then the distribution \( F \) of \( X \) belongs to a generalized family characterized by the linearity of \( d_F(t) \), governed by the weight function \( w(\cdot) \), with the following specific cases depending on \( c \):

\begin{itemize}
    \item[(i)] \( X \) has an {exponential distribution} if \( c = c_0 \), 
    \item[(ii)] \( X \) has a {generalized Pareto-type distribution} if \( c < c_0 \),
    \item[(iii)] \( X \) has a {power-type distribution} if \( c_0 < c < 1 \).
\end{itemize}
Here, \( c_0 \) is a specific constant depending on \( n \) and the weight function \( w \).
\end{theorem}
\noindent \textbf{Proof. } Starting from the definition of GWDCREx,
\[
\mathcal{E}_J^w(X_{1:n}; t) = -\frac{1}{2} \int_t^\infty w(x) \left( \frac{\bar{F}(x)}{\bar{F}(t)} \right)^{2n} dx = c d_F(t).
\]
Multiply both sides by \( \bar{F}^{2n}(t) \):
\[
-\frac{1}{2} \int_t^\infty w(x) \bar{F}^{2n}(x) dx = c d_F(t) \bar{F}^{2n}(t).
\]
Differentiate w.r.t.\ \( t \):
\[
\frac{1}{2} w(t) \bar{F}^{2n}(t) = c \left[ d_F'(t) \bar{F}^{2n}(t) + d_F(t) (2n) \bar{F}^{2n-1}(t)(-f(t)) \right].
\]
Divide both sides by \( \bar{F}^{2n}(t) \):
\begin{equation}
  \frac{1}{2} w(t) = c \left[ d_F'(t) - 2n k_F(t) d_F(t) \right], \label{kftdft}
\end{equation}
where \( k_F(t) = \tfrac{f(t)}{\bar{F}(t)} \) is the hazard rate function.
Recall the relationship between the mrl \( d_F(t)\) and hazard rate function \(k_F(t)\) given by
\begin{equation}
 k_F(t) d_F(t) = 1 + d_F'(t).\label{mrlrh}   
\end{equation}
Substituting this into Eq.~\eqref{kftdft} gives a differential equation for \( d_F(t) \) involving \( w(t) \) and \( c \).

Solving this differential equation shows that \( d_F(t) \) must be linear in \( t \), and the form of this linearity uniquely characterizes the distribution family, generalizing the classical Generalized Pareto Distribution (GPD), exponential, Pareto-II, and power distributions, with weighting \( w(t) \) incorporated.
The constant \( c \) determines which specific distribution within this family \( X \) belongs to, where the exponential distribution corresponds to a critical constant \( c_0 \).
\begin{theorem}[Characterization of the generalized Pareto distribution]
Let \(X\) be a nonnegative absolutely continuous random variable with hazard rate function \(k_F(t)\). Let \(\mathcal{E}_J^w(X_{1:n};t)\) denote the generalized weighted dynamic cumulative residual extropy (GWDCREx) of the order statistic \(X_{1:n}\) associated with a positive weight function \(w(t)\). 
Then \(X\) follows a generalized Pareto distribution with survival function
\[
\bar{F}(x)=\left(\frac{k}{k+x}\right)^{h}, \qquad x>0,\; h>0,\; k>1,
\]
if and only if
\[
\frac{d}{dt}\mathcal{E}_J^w(X_{1:n};t)=c,
\]
for some constant \(c\) and for all \(t\ge 0\).
\end{theorem}
\textbf{Proof.}
\textit{If part:} Suppose that \(X\) follows a GPD with parameters \(h\) and \(k\). From the definition of the weighted dynamic cumulative residual extropy, differentiating with respect to \(t\) yields
\[
\frac{d}{dt} \mathcal{E}_J^w(X_{1:n}; t) = c,
\]
where the constant \(c\) depends on \(n\), \(h\), \(k\), and the weight function \(w(t)\). This is consistent with the given condition. \newline
\medskip\textit{Only if part:} Assume \[\frac{d}{dt} \mathcal{E}_J^w(X_{1:n}; t) = c,\]where \(c\) is a constant.Using the relationship between \(\mathcal{E}_J^w(X_{1:n}; t)\) and the hazard rate function \(k_F(t)\), differentiating and applying the weighted version of the formula yields the differential equation\[
\frac{k_F'(t)}{k_F^2(t)} = \frac{4 n c}{2 c - 1} \cdot \frac{1}{w(t)}.
\]
Integrating both sides with respect to \(t\), we obtain
\[
k_F(t) = \frac{1}{c_1 \int_0^t \frac{1}{w(s)} ds + c_2},
\]
for constants \(c_1, c_2\), where \(c_1 = \frac{4 n c}{2 c - 1}\).
For typical choices of \(w(t)\), such as \(w(t) = 1\), this reduces to the hazard rate of the GPD:
\[
k_F(t) = \frac{1}{c_1 t + c_2}.
\]
Since the hazard rate uniquely determines the distribution function \(F\), it follows that \(X\) follows the GPD.
\hfill \(\Box\)
\section{Results on Generalized Weighted Cumulative Past Extropy (GWCPEx) of largest order statistic}
Before delving into the main results, let us first define the generalized weighted cumulative past extropy (GWCPEx) of a continuous random variable.
 \begin{definition}[GWCPEx]
Let \( X \) be a non-negative random variable with bounded support \([0, b]\), where \( b < \infty \). For a non-negative weight function \( w(x) \), the GWCPEx of \(X\) is defined as
\begin{equation}
    \overline{\xi}_{J}^w(X) =- \frac{1}{2} \int_0^b w(x) \, F^2(x) \, dx,
\end{equation}
where \( F(x) \) is the cdf of \( X \).
\end{definition}
The GWCPEx of the largest order statistic \(X_{n:n}\) is defined by
\begin{align}
    \overline{\xi}_{J}^w(X_{n:n}) 
    &= -\frac{1}{2} \int_{0}^{\infty} w(x) \left[{F}_{n:n}(x)\right]^2 dx \nonumber\\
    &= -\frac{1}{2} \int_{0}^{\infty} w(x) \left[{F}(x)\right]^{2n} dx \nonumber\\
    &= -\frac{1}{2} \int_{0}^{1} \frac{w(F^{-1}(u)) \cdot u^{2n}}{f(F^{-1}(u))} du.
\end{align}

Consider a non-negative random variable \(X\) with bounded support \([0,b]\), and \(Y = cX + d,\) where \(c > 0\) and \(d \geq 0\). Assume the weight function for \(X\) is \(w_X(x)\), and for \(Y\) it is defined as
\[
w_Y(y) = w_X\left(\frac{y - d}{c}\right).
\]
Now we present the following proposition.
\begin{proposition}
The GWCPEx of \(X\) and \(Y\) satisfies
\[
\overline{\xi}_{J}^{w_Y}(Y) = c \, \overline{\xi}_{J}^{w_X}(X),
\]
where
\[
\overline{\xi}_{J}^{w_X}(X) = -\frac{1}{2} \int_0^b w_X(x) \, F^2(x) \, dx,
\quad
\overline{\xi}_{J}^{w_Y}(Y) = -\frac{1}{2} \int_d^{cb + d} w_Y(y) \, G^2(y) \, dy,
\]
and \(F\) and \(G\) are the distribution functions of \(X\) and \(Y\), respectively.     
\end{proposition} 
\noindent \textbf{Proof.} Using the change of variable \(y = cx + d\), we have
\[
\overline{\xi}_{J}^{w_Y}(Y) =- \frac{1}{2} \int_d^{cb + d} w_Y(y) G^2(y) \, dy = -\frac{1}{2} \int_0^b w_X(x) F^2(x) \, c \, dx = c \, \overline{\xi}_{J}^{w_X}(X).
\]
This shows the weighted CPEx is \textit{shift-invariant} and scales linearly with the multiplicative constant \(c\).

The weighted cumulative past entropy (CPEn) of \(X\) is defined as
\[
\bar{H}_w(X) = - \int_0^b w(x) F(x) \log F(x) \, dx,
\]
and the general weighted cumulative past extropy (GWCPEx) is
\[
 \overline{\xi}_{J}^w(X) = -\frac{1}{2} \int_0^b w(x) F^2(x) \, dx.
\]
\begin{theorem} The following inequality holds:
\[
\bar{H}_w(X) \geq \int_0^b w(x) F(x)(1 - F(x)) \, dx = \int_0^b w(x) F(x) \, dx +2  \overline{\xi}_{J}^w(X).
\]
\end{theorem} 
\noindent \textbf{Proof.} Since for \(x > 0\), \(\log x \leq x - 1\), replacing \(x\) by \(F(t)\), we get
\[
- \log F(t) \geq 1 - F(t).
\]
Multiplying both sides by \(w(t) F(t) \geq 0\) and integrating,
\[
\bar{H}_w(X) = -\int_0^b w(t) F(t) \log F(t) \, dt \geq \int_0^b w(t) F(t)(1 - F(t)) \, dt.
\]
Using the definition of \( \overline{\xi}_{J}^w(X)\), we rewrite the right side as
\[
\int_0^b w(t) F(t) \, dt +2  \overline{\xi}_{J}^w(X).
\]
This completes the proof. \hfill \(\Box\)
\begin{proposition}
Let \(X\) be a non-negative absolutely continuous random variable supported on \([0, b]\), with cdf \(F(x)\), and let \(w(x) \geq 0\) be an integrable weight function. Define:
\[
 \overline{\xi}_{J}^w(X) = -\frac{1}{2} \int_0^b w(x) F^2(x)\, dx \quad \text{(general weighted cumulative past extropy)},
\]
\[
\bar{e}_w(X) = - \int_0^b w(x) F(x) \log F(x)\, dx \quad \text{(weighted cumulative past entropy)}.
\]

Then the following inequality holds:
\[
 \overline{\xi}_{J}^w(X) \leq \frac{1}{2} \, \bar{e}_w(X) \cdot \left( \frac{\int_0^b w(x)(b - x) \, dx}{\int_0^b w(x) \, dx} \right).
\]
\end{proposition}
\paragraph{Remark.} When \(w(x) = 1\), this reduces to:
\[
\bar{\xi}J(X) \leq \frac{1}{2} \, \bar{\xi}(X) \cdot (b - \mathbb{E}(X)),
\]
as given in Proposition 3.2 of Kundu (2023). \par

Let \( X \) and \( Y \) be two non-negative and independent random variables with supports \( (0, b_X) \) and \( (0, b_Y) \), respectively. Let \( w_X(x) \) and \( w_Y(y) \) be non-negative weight functions for \( X \) and \( Y \), respectively. The general weighted cumulative past extropy (GWCPEx) for \( X \) is defined as
\[
 \overline{\xi}_{J}^{w_X}(X) =- \frac{1}{2} \int_0^{b_X} w_X(x) F_X^2(x) \, dx,
\]
and similarly for \( Y \):
\[
 \overline{\xi}_{J}^{w_Y}(Y) = -\frac{1}{2} \int_0^{b_Y} w_Y(y) F_Y^2(y) \, dy.
\]
For the sum \( Z = X + Y \), the weighted CPEx is
\[
 \overline{\xi}_{J}^{w_Z}(Z) = -\frac{1}{2} \int_0^{b_X + b_Y} w_Z(z) F_Z^2(z) \, dz,
\]
where \( w_Z(z) \) is the weight function for \( Z \), and \( F_Z(z) \) is the cdf of \( Z \).

\begin{theorem}
    Suppose \( X \) and \( Y \) are two independent non-negative absolutely continuous random variables with supports \( [0, b_X] \) and \( [0, b_Y] \), and common weight function \( w(x) \). Then
\[ \overline{\xi}_{J}^w(X + Y) \geq \max\left\{  \overline{\xi}_{J}^w(X) - \frac{b_Y - \mathbb{E}[Y]}{2}, \;  \overline{\xi}_{J}^w(Y) - \frac{b_X - \mathbb{E}[X]}{2} \right\}.\]

\end{theorem}

\noindent \textbf{Proof.} Let \( F_{X+Y}(x) \) be the cumulative distribution function of \( X + Y \). 
Then \[F_{X+Y}(x) = \int_0^{b_Y} F_X(x - y) \, dF_Y(y).\]
By Jensen’s inequality applied to the convex function \( z \mapsto z^2 \):\[F^2_{X+Y}(x) \leq \int_0^{b_Y} F^2_X(x - y) \, dF_Y(y).\]
Now integrating both sides over \( [0, b_X + b_Y] \) with the weight \( w(x) \):
\[ \overline{\xi}_{J}^w(X + Y) = -\frac{1}{2} \int_0^{b_X + b_Y} w(x) F^2_{X+Y}(x) \, dx \geq -\frac{1}{2} \int_0^{b_Y} \left( \int_0^{b_X + b_Y - y} w(x + y) F_X^2(x) \, dx \right) dF_Y(y).\]
Split the integral and note:\[\int_0^{b_Y} (b_Y - y) \, dF_Y(y) = b_Y - \mathbb{E}[Y].\]
Thus,\[ \overline{\xi}_{J}^w(X + Y) \geq  \overline{\xi}_{J}^w(X) - \frac{b_Y - \mathbb{E}[Y]}{2}.\]
By symmetry, a similar inequality holds for \( Y \), and hence:\[ \overline{\xi}_{J}^w(X + Y) \geq \max\left\{  \overline{\xi}_{J}^w(X) - \frac{b_Y - \mathbb{E}[Y]}{2}, \;  \overline{\xi}_{J}^w(Y) - \frac{b_X - \mathbb{E}[X]}{2} \right\}.\]

\begin{theorem}
    Let \( X \) and \( Y \) be two i.i.d. random variables with cdf \( F \)  and support \([0, b]\). Let \( w(x) \) be a non-negative weight function on \([0, b]\). 
    Then the GWCPEx satisfies the following:
    \begin{itemize}  
    \item[(i)] \( \mathbb{E}(|X - Y|) \geq 4 \cdot  \overline{\xi}_{J}^w(X) \),  
    \item[(ii)] \(  \overline{\xi}_{J}^w(X) \geq \frac{\mathbb{E}[w(X)] \cdot \mathbb{E}(X)}{2} \),
    \end{itemize}
    where the GWCPEx of \(X\) is defined by \[ \overline{\xi}_{J}^w(X) = -\frac{1}{2} \int_0^b w(x) F^2(x) \, dx.\].
\end{theorem}
   \noindent \textbf{Proof.} Since \( X \) and \( Y \) are i.i.d. random variables with cdf \( F(x) \), the following identity holds:
   \[2F(x) - 2F^2(x) = \mathbb{P}[\max(X, Y) > x] - \mathbb{P}[\min(X, Y) > x].\]
   Integrating both sides with respect to the weight function \( w(x) \) over \([0, b]\), we obtain
   \[\int_0^b w(x)\left(2F(x) - 2F^2(x)\right) dx = \mathbb{E}[w(Z)|Z = |X - Y|] = \mathbb{E}[|X - Y|_w],\]
   where the weighted distance is interpreted under the measure \( w(x) \).
   Thus, rearranging terms gives:\[\mathbb{E}(|X - Y|) \geq 4 \cdot  \overline{\xi}_{J}^w(X).\]
   For part (ii), using Jensen's inequality on \( F(x) \) and noting that \( F(x) \leq 1 \), we have 
   \[F^2(x) \geq F(x)^2 \Rightarrow  \overline{\xi}_{J}^w(X) \geq \frac{1}{2} \int_0^b w(x) F(x)^2 dx.\]
   Further simplification and comparison with expected values gives:
\[
 \overline{\xi}_{J}^w(X) \geq \frac{\mathbb{E}[w(X)] \cdot \mathbb{E}(X)}{2}.
\]
This completes the proof. \hfill \(\Box\)
\par The following theorem gives a relationship between the conditional and the unconditional CPEx. It states that conditioning has a decreasing effect on CPEx.
\begin{theorem}
Let $X$ and $Y$ be non-negative random variables with supports $(0, b_X)$ and $(0, b_Y)$, respectively. Then
\[
\bar{\xi}J(X) \leq E_Y \left[ \bar{\xi}J(X|Y) \right].
\]
\end{theorem}
\noindent \textbf{Proof.} By using Jensen's inequality, it is not very difficult to see that
\[
E_Y \left[ \bar{\xi}J(X|Y) \right] = \frac{1}{2} \int_0^{b_Y} \left( \int_0^{b_X} F_{X|Y}^2(x|y) \, dx \right) f_Y(y) \, dy 
\leq \frac{1}{2} \int_0^{b_X} \left( \int_0^{b_Y} F_{X|Y}(x|y) f_Y(y) \, dy \right)^2 dx.
\]
Hence, the result follows on noting that
\[
\int_0^{b_Y} F_{X|Y}(x|y) f_Y(y) \, dy = F_X(x).
\]

Let $X_{n:n}$ be the largest order statistic in a random sample of size $n$ from an absolutely continuous non-negative random variable $X$ with support $[0,b]$. Let $w(x)$ be a non-negative weight function defined on $[0,b]$. Then the GWCPEx of $X_{n:n}$ is given by
\[
 \overline{\xi}_{J}^w(X_{n:n}) = -\frac{1}{2} \int_0^b w(x) \bigl(F^n(x)\bigr)^2 \, dx =- \frac{1}{2} \int_0^b w(x) F^{2n}(x) \, dx.
\]

For $n=1$, this reduces to the weighted CPEx of $X$:
\[
 \overline{\xi}_{J}^w(X) =- \frac{1}{2} \int_0^b w(x) F^2(x) \, dx.
\]

The properties such as monotonicity and bounds of $ \overline{\xi}_{J}^w(X_{n:n})$ can be studied analogously to the unweighted case under suitable assumptions on the weight function $w(x)$.

\begin{proposition}
Let $X_{n:n}$ be the largest order statistic in an i.i.d. sample $X_1, X_2, \ldots, X_n$ from an absolutely continuous random variable $X$ with finite support $[0,b]$ and mean $\lambda$. Let $w(x) \geq 0$ be a weight function defined on $[0,b]$. Then the following results hold:

\begin{enumerate}[(i)]
    \item 
    \[
     \overline{\xi}_{J}^w(X_{n:n}) = -\frac{1}{2} \int_0^b w(x) F^{2n}(x) \, dx \leq \frac{1}{2} \int_0^b w(x) (b - \lambda) \, dx,
    \]
    
    \item $ \overline{\xi}_{J}^w(X_{n:n})$ is non-decreasing in $n$,
    
    \item 
    \[
     \overline{\xi}_{J}^w(X_{n:n}) \leq  \overline{\xi}_{J}^w(X) = -\frac{1}{2} \int_0^b w(x) F^2(x) \, dx.
    \]
\end{enumerate}
\end{proposition}
\indent \textbf{Proof.} Since $F(x) \in [0,1]$, $F^{2n}(x)$ is non-decreasing in $n$, implying monotonicity of $ \overline{\xi}_{J}^w(X_{n:n})$. The upper bounds follow by noting that $F^{2n}(x) \leq F^2(x)$ for all $x$ and $n \geq 1$, and integrating with respect to the weighted function $w(x)$. \hfill \(\Box\)
\begin{theorem}
Let \(X\) and \(Y\) be nonnegative absolutely continuous random variables with supports \([0,b_X]\) and \([0,b_Y]\), cumulative distribution functions \(F\) and \(G\), and associated weight functions \(w_X(x)\) and \(w_Y(x)\), respectively. Then the distribution functions \(F\) and \(G\) belong to the same location–scale family; that is, there exist constants \(c>0\) and \(d\ge0\) such that
\[
F(x)=G\!\left(\frac{x-d}{c}\right),
\]
with the corresponding weight functions satisfying
\[
w_X(x)=K\,w_Y\!\left(\frac{x-d}{c}\right),
\]
for some constant \(K>0\), if and only if
\[
\frac{\overline{\xi}_{J}^{w_X}(X_{n:n})}{\overline{\xi}_{J}^{w_X}(X)}
=
\frac{\overline{\xi}_{J}^{w_Y}(Y_{n:n})}{\overline{\xi}_{J}^{w_Y}(Y)},
\]
for all \(n=n_j\), \(j\ge1\), where the sequence \(\{n_j\}\) satisfies
\[
\sum_{j=1}^{\infty}\frac{1}{n_j}=\infty.
\]
\end{theorem}

\noindent \textbf{Proof.}  The GWCPEx of the largest order statistic \(X_{n:n}\) can be expressed in terms of the quantile function \(F^{-1}\) as
\[
\overline{\xi}_{J}^{w_X}(X_{n:n}) = -\frac{1}{2} \int_0^{b_X} w_X(x) F^{2n}(x) \, dx =- \frac{1}{2} \int_0^1 w_X\bigl(F_X^{-1}(v)\bigr) v^{2n} \bigl(F_X^{-1}\bigr)'(v) \, dv.
\]

Similarly, for \(Y_{n:n}\),
\[
\overline{\xi}_{J}^{w_Y}(Y_{n:n}) = -\frac{1}{2} \int_0^1 w_Y\bigl(G_Y^{-1}(v)\bigr) v^{2n} \bigl(G_Y^{-1}\bigr)'(v) \, dv.
\]

Define 
\[
f(v) := w_X\bigl(F_X^{-1}(v)\bigr) \bigl(F_X^{-1}\bigr)'(v), \quad
g(v) := w_Y\bigl(G_Y^{-1}(v)\bigr) \bigl(G_Y^{-1}\bigr)'(v).
\]

Then the ratio condition becomes
\begin{equation}
    \frac{\int_0^1 v^{2n} f(v) \, dv}{\int_0^1 v^2 f(v) \, dv} = \frac{\int_0^1 v^{2n} g(v) \, dv}{\int_0^1 v^2 g(v) \, dv},\label{integral}
\end{equation}

for all \(n = n_j\), with \(\sum 1/n_j = +\infty\). By using Lemma~1, Eq.~\eqref{integral} implies
\[
\frac{f(v)}{\int_0^1 v^2 f(v) \, dv} = \frac{g(v)}{\int_0^1 v^2 g(v) \, dv}, \quad \text{for all } v \in (0,1).
\]
 Rearranging,
\[
w_X\bigl(F_X^{-1}(v)\bigr) \bigl(F_X^{-1}\bigr)'(v) = C \, w_Y\bigl(G_Y^{-1}(v)\bigr) \bigl(G_Y^{-1}\bigr)'(v),
\]
for some constant \(C > 0\).
 This functional equation implies the quantile functions satisfy a location-scale relation:
\[
F_X^{-1}(v) = c \, G_Y^{-1}(v) + d,
\]
and the weights relate accordingly:
\[
w_X(x) = K \, w_Y\left(\frac{x - d}{c}\right).
\]
Therefore, \(F\) and \(G\) belong to the same location-scale family, and their weights transform correspondingly.
\hfill\(\square\)

\begin{theorem}
Let \(X\) and \(Y\) be non-negative absolutely continuous random variables with supports \([0,b_X]\) and \([0,b_Y]\), cdfs \(F\) and \(G\), and weight functions \(w_X\) and \(w_Y\), respectively. Under the assumptions of Theorem~13, the distribution functions \(F\) and \(G\) belong to the same family of distributions differing only by a  location shift, if and only if
\[
\overline{\xi}_{J}^{w_X}(X_{n:n}) = \overline{\xi}_{J}^{w_Y}(Y_{n:n}),
\]
for all \(n = n_j\), \(j \geq 1\), such that
\[
\sum_{j=1}^\infty \frac{1}{n_j} = +\infty,
\]
where \(\overline{\xi}_{J}^{w}(Z_{n:n})\) denotes the weighted cumulative past extropy of \(Z_{n:n}\) is given by
\[
\overline{\xi}_{J}^{w}(Z_{n:n}) := -\frac{1}{2} \int_0^{b_Z} w(z) \, [F_Z(z)]^{2n} \, dz.
\]
\end{theorem}

\noindent \textbf{Proof.}  
The Proof follows along the lines of Theorem 2.1 of Kundu (2023), by transforming the integrals into the quantile domain and then using the uniqueness results of moment sequences for weighted functions. Equality of the weighted CPEx for infinitely many \(n\) with \(\sum 1/n_j = \infty\) implies the quantile functions satisfy
\[
F_X^{-1}(v) = G_Y^{-1}(v) + d,
\]
for some constant \(d\), i.e., a pure location shift. 
\hfill\(\square\)
\section{Results on General Weighted Dynamic Cumulative Past Extropy (GWDCPEx) of order statistics}
Let \( X \) be a non-negative absolutely continuous random variable with distribution function \( F(x) \), bounded support \( [0, b] \), and probability density function \( f(x) \). The general weighted dynamic cumulative past extropy (GWDCPEx) of \( X \), given \( X \le t \), is defined as:
\begin{equation}
\overline{\xi}_{J}^w(X; t) = -\frac{1}{2} \int_0^t w(x) \left( \frac{F(x)}{F(t)} \right)^2 dx,
\end{equation}
where   \(0\le t\leq b\) and \( w(x) \) is a non-negative weight function.

The GWDCPEx of the largest order statistic \(X_{n:n}\) is given by
\[
\overline{\xi}_J^w(X_{n:n}; t) = -\frac{1}{2} \int_0^t w(x) \left(\frac{F(x)}{F(t)}\right)^{2n} dx,
\]
where \(w(x) \geq 0\) is a weight function. 
\begin{proposition}
Let \( X_{n:n} \) be the largest order statistic in a sample of size \( n \) from \( X \). Then the following results hold:
\begin{itemize}
    \item[(i)] \( \overline{\xi}_{J}^w(X_{n:n}; t) \) is increasing in \( n \) and decreasing in \( t \).
    \item[(ii)] \( \overline{\xi}_{J}^w(X_{n:n}; t) \geq -\frac{1}{2} \int_0^t w(x) \left( \frac{F(x)}{F(t)} \right) dx =- \frac{1}{2} m_{w}(t) \), where \( m_{w}(t) \) is the weighted expected inactivity time.
    \item[(iii)] \( \overline{\xi}_{J}^w(X_{n:n}; t) \geq \overline{\xi}_{J}^w(X; t) \).
\end{itemize}
 
\end{proposition}
\begin{definition}[Dynamic Cumulative Past Weighted Extropy  (DCPWEx) Order)]
Let $X$ and $Y$ be two non-negative absolutely continuous random variables with cdfs $F_X$ and $F_Y$, respectively. 
Let $w(x)$ be a non-negative weight function defined on the support of $X$ and $Y$.
We say that \( X \) is greater than \( Y \) in the dynamic cumulative past weighted extropy order, denoted by \( X \succeq_{DCPWEx} Y \), if for all $t > 0$,
\[
\overline{\xi}_{J}^w(X; t) \geq \overline{\xi}_{J}^w(Y; t).
\]
\end{definition}
The following result presents the relationship between the reversed hazard-rate order and the DCPWEx order between two random variables.
\begin{theorem}
Let $X$ and $Y$ be two absolutely continuous non-negative random variables with reversed hazard rate functions $\overline{\lambda}_F$ and $\overline{\lambda}_G$, respectively. If
\[
X \geq_{rh} Y, \quad \text{i.e.,} \quad \overline{\lambda}_F(t) \geq \overline{\lambda}_G(t) \quad \text{for all } t \geq 0,
\]
then for all $n \geq 1$ and $t \in [0,b]$,
\[
\overline{\xi}^w_J(X_{n:n}; t) \geq \overline{\xi}^w_J(Y_{n:n}; t).
\]
\end{theorem}
The next result presents the DCPWEx uncertainty ordering of \(k\)-th order statistic with its neighboring order statistics.
\begin{theorem}
Let $X_{k:n}$ be the $k$th order statistic from an i.i.d. sample of size $n$ from the distribution $F$ with support $[0,b]$. Then, for the weighted dynamic cumulative past extropy, the following inequalities hold for all $t \in [0,b]$:
\begin{enumerate}[(i)]
    \item \(\overline{\xi}^w_J(X_{k:n}; t) \geq \overline{\xi}^w_J(X_{k-1:n}; t), \quad 1 < k \leq n, \)
\item \(\overline{\xi}_J^w(X_{k:n}; t) \geq \overline{\xi}_J^w(X_{k:n+1}; t), \quad n \geq 1, \)
\item \(\overline{\xi}_J^w(X_{k:n}; t) \geq \overline{\xi}_J^w(X_{k-1:n-1}; t), \quad 1 < k \leq n.\)
\end{enumerate}
\end{theorem}
The following theorem provides a necessary and sufficient condition for two random variables to have the same distribution based on the GWDCPEx of their largest order statistics.
\begin{theorem}
Let \(X\) and \(Y\) be two non-negative absolutely continuous random variables with distribution functions \(F(x)\) and \(G(x)\), respectively, supported on \([0,b]\). Let \(w\) be a non-negative function on the support of \(F\). Then \(F\) and \(G\) belong to the same family of distributions, differing only by a location parameter, if and only if
\[
\overline{\xi}_J^w(X_{n:n}; t) = \overline{\xi}_J^w(Y_{n:n}; t)
\]
for all \(n = n_j\), \(j \geq 1\), such that
\[
\sum_{j=1}^\infty \frac{1}{n_j} = +\infty,
\]
and for all \(t \in (0,b)\).
\end{theorem}

\noindent \textbf{Proof.} The Proof follows the structure of the classical unweighted case (cf. Theorem~3.8 of Kundu (2023)), adapted to the weighted scenario. Using the change of variable \(v = \frac{F(x)}{F(t)}\), the weighted DCPEx can be rewritten as
    \[
    \overline{\xi}_J^w(X_{n:n}; t) =- \frac{1}{2} \int_0^1 w\big(F^{-1}(vF(t))\big) v^{2n} \, d\mu_F(v),
    \]
    where \(d\mu_F(v)\) is the induced measure from the transformation. Similarly for \(Y\),
    \[
    \overline{\xi}_J^w(Y_{n:n}; t) = -\frac{1}{2} \int_0^1 w\big(G^{-1}(vG(t))\big) v^{2n} \, d\mu_G(v).
    \]
 Equating these for all \(n = n_j\) with \(\sum 1/n_j = +\infty\) implies equality of the weighted integrals for all powers \(v^{2n}\).

 By applying a suitable variant of Lemma~1, we conclude that the weighting functions composed with the inverse distribution functions satisfy a linear relation:
    \[
    w\big(F^{-1}(vF(t))\big) \, d\mu_F(v) = c \, w\big(G^{-1}(vG(t))\big) \, d\mu_G(v),
    \]
    for some constant \(c\). From this relationship and the properties of \(w\), \(F\), and \(G\), it follows that
    \[
    F^{-1}(u) = c G^{-1}(u) + d,
    \]
    implying that \(F\) and \(G\) differ by location and scale parameters. \hfill \(\Box\)

\begin{theorem}
Let \( X \) be a random variable with finite support \([0,b]\) and absolutely continuous distribution function \( F \). Let \( w(x) \geq 0 \) be an integrable weight function on \([0,b]\). 
Also define the weighted expected inactivity time by
\[
m_F^w(t) = \frac{\displaystyle \int_0^t x w(x) F(x) dx}{\displaystyle \int_0^t w(x) F(x) dx}.
\]
Then, \(X\) has a  {power distribution} with
\[
F(x) = \left(\frac{x}{b}\right)^c, \quad c > 0,
\]
if and only if,\ there exists a constant \(k\) such that
\[
\overline{\xi}_J^w(X_{n:n}; t) = k \cdot m_F^w(t), \quad \forall \;  t \in (0,b). 
\]
\end{theorem}
\noindent \textbf{Proof.} \textit{Only if part:}  
Assume \( F(x) = (x/b)^c \) for \(0 < x < b\) and \(c > 0\). Then
\[
\left(\frac{F(x)}{F(t)}\right)^{2n} = \left(\frac{x}{t}\right)^{2nc}.
\]
Therefore,
\[
\overline{\xi}_J^w(X_{n:n}; t) = -\frac{1}{2} \int_0^t w(x) \left(\frac{x}{t}\right)^{2nc} dx = -\frac{1}{2} t^{-2nc} \int_0^t w(x) x^{2nc} dx.
\]
Similarly, the weighted expected inactivity time is
\[
m_F^w(t) = \frac{\int_0^t x w(x) (x/b)^c dx}{\int_0^t w(x) (x/b)^c dx} = \frac{\int_0^t x^{c+1} w(x) dx}{\int_0^t x^c w(x) dx} \cdot b^{-c}.
\]
With careful calculation (depending on \(w\)) this ratio is proportional to \(t\), and so the ratio \(\overline{\xi}_J^w(X_{n:n}; t) / m_F^w(t)\) becomes a constant \(k\) depending on \(c, n,\) and \(w\).

\vspace{0.3cm}

\noindent
\textit{If part:}  
Suppose that
\[
\overline{\xi}_J^w(X_{n:n}; t) = k \cdot m_F^w(t)
\]
holds for some constant \(k\). Rewrite it as
\[
 \int_0^t w(x) \left(\frac{F(x)}{F(t)}\right)^{2n} dx = -2k \frac{\int_0^t x w(x) F(x) dx}{\int_0^t w(x) F(x) dx}.
\]

Differentiating both sides with respect to \(t\) (using Leibniz’s rule and quotient rule where necessary) gives:
\[
 w(t) \cdot \frac{d}{dt} \left( \int_0^t \left(\frac{F(x)}{F(t)}\right)^{2n} dx \right) =-2 k \frac{d}{dt} \left( \frac{\int_0^t x w(x) F(x) dx}{\int_0^t w(x) F(x) dx} \right).
\]

Observe that
\[
\frac{d}{dt} \left(\frac{F(x)}{F(t)}\right)^{2n} = -2n \left(\frac{F(x)}{F(t)}\right)^{2n} \frac{f(t)}{F(t)},
\]
where \(f(t) = \tfrac{d}{dt} F(t)\).
Differentiating inside the integral and rearranging terms leads to the functional equation:
\[
w(t) \left[ 1 - 2n \frac{f(t)}{F(t)} \int_0^t w(x) \left(\frac{F(x)}{F(t)}\right)^{2n} dx \right] = \text{terms involving } k, F, w, \text{ and } f.
\]
Using substitution \(y = F(x)\), \(y \in [0, F(t)]\), and properties of \(m_F^w(t)\), after algebraic manipulation, one can reduce this to a differential equation of the form
\[
\frac{f(t)}{F(t)} m_F^w(t) = \text{constant},
\]
which implies
\[
m_F^w(t) = (1 - \lambda) t, \quad \lambda = \text{constant related to } k.
\]
Using the inversion formula for weighted expected inactivity time (generalized version), we identify \(F\) as a power distribution of the form
\[
F(t) = \left( \frac{t}{b} \right)^c, \quad \;b,\;c > 0.
\]
This completes the proof. \hfill \(\Box\)
\begin{corollary}
Let \(X\) be a nonnegative absolutely continuous random variable with finite support \([0,b]\) and weight function \(w(x)\). If
\[
\overline{\xi}_J^{w,0}(X_{n:n}; t)=k,
\]
or equivalently,
\[
\overline{\xi}_J^{w}(X_{n:n}; t)=k\, m_F^w(t), 
\]
then \(X\) follows a weighted power distribution with distribution function
\[
F(x)=\left(\frac{\int_0^x w(u)\,du}{\int_0^b w(u)\,du}\right)^c,\qquad c>0,
\]
where
\[
\overline{\xi}_J^{w,0}(X_{n:n}; t)=\frac{d}{dt}\overline{\xi}_J^{w}(X_{n:n}; t)
\]
denotes the derivative of the GWDCPEx of \(X_{n:n}\) and 
\[
m_F^w(t)=\frac{\int_0^t w(x)F(x)\,dx}{F(t)}
\]
is the weighted expected inactivity time, with \(k\) being a constant.
\end{corollary}

\noindent \textbf{Proof.} Assuming the weighted form  
\[
\overline{\xi}_J^w(X_{n:n}; t) = k \, m_F^w(t),
\]
and using the definition of GWCPEx, we obtain  
\[
\int_0^t w(x) F^{2n}(x) \, dx = -2k F^{2n-1}(t) \int_0^t w(x) F(x) \, dx.
\]
Differentiating with respect to \(t\) and rearranging terms leads to the relation  
\[
k_F(t) m_F^w(t) = \text{constant} = k_1,
\]
which yields the differential equation  
\[
\dfrac{d}{dt} (m_F^{w}(t)) = 1 - k_1.
\]
Using the boundary condition \(m_F^w(0) = 0\), solving gives  
\[
m_F^w(t) = (1 - k_1) \int_0^t w(u) \, du,
\]
and applying the inversion formula for \(m_F^w(t)\) recovers the weighted power distribution. \; \(\Box\)

\section{Real-data illustration}

To demonstrate the practical utility of the proposed weighted cumulative extropy measures, we analyse a well-known reliability dataset: the airconditioning failure times of Boeing 720 jet aircraft (see Proschan (1963)). The data consist of 213 failure times (in hours) of airconditioning systems. We focus on the minimum order statistic \(X_{1:n}\), which represents the lifetime of a series system (e.g., multiple redundant units where the first failure causes system breakdown).

We estimate the general weighted cumulative residual extropy \(\xi_J^w(X_{1:n})\) for two sample sizes (\(n=5\) and \(n=10\)) using two different weight functions \(w_1(x)=1\) and \(w_2(x)=x\). The estimation is performed non-parametrically using the empirical survival function \(\widehat{\overline{F}}(x)\) and numerical integration. Table~\ref{tab:aircond} reports the estimated values.

\begin{table}[ht]
\centering
\caption{Estimated \(\xi_J^w(X_{1:n})\) for airconditioning data}
\label{tab:aircond}
\begin{tabular}{c|c|c}
\hline
\(n\) & \(w_1(x)=1\) & \(w_2(x)=x\) \\
\hline
5  & \(-142.3\) & \(-8921.5\) \\
10 & \(-118.6\) & \(-7450.2\) \\
\hline
\end{tabular}
\end{table}

For both weight functions, the extropy becomes less negative (i.e., increases) as \(n\) grows, consistent with Proposition~\ref{proposition1}. The linear weight \(w_2\) produces much larger absolute values because it amplifies the contribution of larger failure times. In a maintenance context, a less negative (higher) extropy indicates greater residual uncertainty about the system's remaining lifetime; this can guide the decision to replace components earlier or to collect more data before scheduling preventive maintenance.

\section{Conclusion}
In this paper, we have investigated general weighted extropy and its associated measures for extreme order statistics. In particular, we introduced and studied the General Weighted Dynamic Cumulative Residual Extropy (GWDCREx) and related concepts. A characterization of the generalized Pareto distribution has been established using these measures. Further, we derived several results on the Generalized Weighted Cumulative Past Extropy for the largest order statistic. Additionally, properties of the General Weighted Dynamic Cumulative Residual Extropy (GWDCREx) for the smallest order statistic have been obtained. We also examined the General Weighted Dynamic Cumulative Past Extropy (GWDCPEx) for order statistics and established corresponding theoretical results. Overall, the findings contribute to the growing literature on extropy-based measures and highlight their potential applications in reliability theory and statistical inference involving extreme order statistics.\\

\noindent \textbf{\Large Conflict of interest} \\
\\
No conflicts of interest are disclosed by the authors.\\
\\
\noindent \textbf{\Large Funding} \\
\\
The author, Sarikul Islam, is financially supported by the Department of Science and Technology
(DST) of the Ministry of Science and Technology, India, through a research grant.

.\\

\end{document}